\documentclass[12pt]{amsart}
\newtheorem{theorem}{Theorem}
\usepackage[b5j,margin=2cm,centering,marginparwidth=1.9cm,marginparsep=.06cm]{geometry}
\makeatletter\def\SK@refcolor{\color{lime}}\def\SK@labelcolor{\color{lime}}\makeatother
\usepackage{graphicx,tikz,pgfplots,epstopdf,amssymb,amsmath,amscd,amsthm,psfrag,epsfig,stackengine,microtype,multirow,cancel,amsfonts,amssymb,color,caption,enumitem,fourier,bm}
\usepackage[mathic=true]{mathtools}
\usepackage[noadjust]{cite}
\usepackage{enumitem}
\usepackage[hidelinks=true]{hyperref}
\usepackage[capitalize,nameinlink,noabbrev]{cleveref}
\newtheorem{proposition}[theorem]{Proposition}
\newtheorem{lemma}{Lemma}

\newtheorem{ltheorem}{Theorem}

\theoremstyle{definition}

\newtheorem{remark}{Remark}
\newif\ifcomments
\newcounter{commentlabel}
  \newlength{\commentlift}
\makeatletter\def\commentfl@g{%
    \vbox to\z@{%
      \setlength{\fboxrule}{0.75pt}%
      \setlength{\fboxsep}{0.75pt}%
      \setlength{\commentlift}{1ex}%
      \addtolength{\commentlift}{\fboxrule}%
      \addtolength{\commentlift}{\fboxsep}%
\vss\color{blue}\rlap{\rlap{\vrule\@height\commentlift\@width\fboxrule}\raise \commentlift%
      \hbox{\fcolorbox{blue}{blue!3}{\normalfont\tiny\bfseries\thecommentlabel}}}}}
  \DeclareRobustCommand{\COMMENT}[1]{\commentstrue\stepcounter{commentlabel}%
\hypertarget{\the\value{section}-\the\value{commentlabel}}{}%
\edef\WRITECOM##1{\noexpand\write\COM{\thecommentlabel, S.##1}}\WRITECOM{\the\value{section}, \noexpand\hyperlink{\the\value{section}-\the\value{commentlabel}}{p.\the\value{page}}: {#1}}\@bsphack%
    \commentfl@g%
    \marginpar{\noindent\lineskip=1pt\lineskiplimit=\maxdimen\raggedright%
    \Tiny\def\Hrule{\strut\hrule height 1pt\strut}\textbf{\color{blue}\llap{\textbullet}\thecommentlabel:}{\color{blue}\thinspace#1\endgraf}}%
  \@esphack}%
  \DeclareRobustCommand{\COMMENTINLINE}[1]{\commentstrue\stepcounter{commentlabel}%
    \write\COM{\thecommentlabel, S.\the\value{section}, p.\the\value{page}: {#1}}\@bsphack%
      \noindent\setlength{\fboxsep}{0pt}\newline\fcolorbox{blue}{lightgray!20}{\begin{minipage}{\textwidth}\tiny\textbf{\color{red}\thecommentlabel}:\thinspace\Tiny\bfseries#1\par\end{minipage}}\allowbreak%
  \@esphack}\makeatother
\newwrite\COM
\immediate\openout\COM=comments
\AtEndDocument{\ifcomments\section*{Marginal Comments}\label{commentstex}{\let\tiny\relax\let\Tiny\relax\let\Hrule\newline\def\pullout#1, S.{\strut\llap{\textbf{#1: }}\S\ }\obeylines\parindent0pt\everypar{\normalsize\normalfont\pullout}\immediate\closeout\COM\def\MT_extended_eqref:n {\eqref}\def\math@bb{\mathbb }\input comments }\fi}
\makeatother
\makeatletter
\newcommand{\st}{\mathbin{\left\bracevert\text{$\phantom{|}$}\right.}}
\newcommand{\dfn}{\mathbin{{\vcentcolon}\hskip-1pt{=}}}

\makeatother
\newcommand{\ie}{that is, }

\newcommand{\Hrule}{{\color{white}\hrule height 2pt}\hrule height 2pt{\color{white}\hrule height 2pt}}
\newsavebox{\topbox}\newlength{\wdoftopbox}
\newsavebox{\botbox}\newlength{\htofbotbox}
\def\Buildrel#1#2{\sbox{\topbox}{#1}\settowidth{\wdoftopbox}{\usebox{\topbox}}\sbox{\botbox}{\ensuremath#2}\settoheight{\htofbotbox}{\usebox{\botbox}}\stackrel{\usebox{\topbox}}{\resizebox{\wdoftopbox}{\htofbotbox}{\usebox{\botbox}}}}
\stackMath
\makeatletter
\newcommand*{\relbarfill@}{\arrowfill@\relbar\relbar\relbar}
\newcommand*{\xminus}[2][]{\ext@arrow 0055\relbarfill@{#1}{#2}}
\newcommand\xxrightarrow[2][]{\mathrel{%
  \setbox2=\hbox{\stackon{\scriptscriptstyle#1}{\scriptscriptstyle#2}}%
  \stackunder[-2pt]{%
    \stackon[.2pt]{\xminus{\makebox[\dimexpr\wd2\relax]{}}}{\scriptscriptstyle#2}%
  }{%
   \scriptscriptstyle#1\,%
  }\,\llap{$\to$}%
}}
\makeatother
\newcommand{\lto}[1]{\xxrightarrow[\;#1\;]{}}
\begin{document}
\title{Horns in billiards}
\author{David De Frutos Ostrander}
\address{D. De Frutos Ostrander, 
Department of Mathematics and Statistics, Queen's University, Kingston, ON K7L 3N6, CA}\email{25xfxh@queensu.ca}
\author{Boris Hasselblatt}
\address{B.~Hasselblatt, Department of Mathematics, Tufts University, Medford, MA 02144, USA}\email{boris.hasselblatt@tufts.edu}
\author{Mark Levi}
\address{M. Levi, Department of Mathematics, The Pennsylvania State University, State College, PA 16802, USA}\email{mxl48@psu.edu}
\thanks{Mark Levi was partially supported by NSF grant DMS-9704554}
\bibliographystyle{plainurlmr}
\begin{abstract}
We show that, like cusps, horns in billiards expel every trajectory after finitely many collisions. We further produce an adiabatic invariant.
\end{abstract}
\maketitle
\section{Introduction}
For planar billiards of a hyperbolic nature, the smooth ergodic theory has been developed in some generality. This is presented well in the definitive book by Chernov and Markarian \cite{ChernovMarkarianBook}, which begins with the set of assumptions under which these developments can be carried out coherently. These assumptions include conditions to exclude an accumulation of collision times \cite[Section 2.4]{ChernovMarkarianBook}. Among them is that there are no horns (\cref{FIGHorn}),\begin{figure}[htb]
	\center{  \includegraphics[width=.5\textwidth]{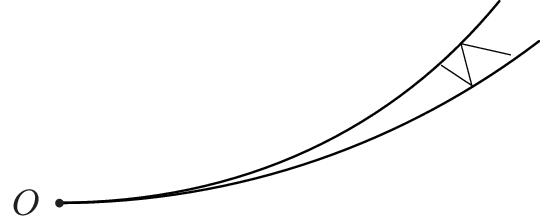}}
	\caption{A horn}  
	\label{FIGHorn}
\end{figure} 
because the authors deem establishing the impossibility of infinitely many collisions in a horn a ``possibly very difficult exercise\rlap.'' At the same time, they are unaware of examples of horns with infinitely many collisions in finite time. The purpose of this note is to establish that horns do not produce infinitely many collisions in finite time. Another aim is to produce an adiabatic invariant (in \cref{THMAdiabatic}).
\begin{ltheorem}\label{thm:main}
As in  \cref{FIGHorn},  consider two $C^3$ arcs in the plane with a common end-point $O$ at which they share a tangent line and have different curvatures of the same sign. Then there exists a neighborhood of $O$ such that any billiard trajectory starting in it will leave this neighborhood after finitely many collisions. 
\end{ltheorem}
The corresponding fact for cusps has long been well known. \cref{thm:main} shows that ``bending a cusp'' so one side becomes focusing does not alter this property, even though the simple geometric argument in \cite[Figure 2.7]{ChernovMakarian} does not apply and a new approach is needed---see \cref{RemarkAngular,REMAngular2} below.

A similar result holds in the presence of gravity \cite{Cusps} (which bends trajectories rather than boundaries).

\section{Proof of Theorem \ref{thm:main}}
We note that the $C^3$ assumption is made to avert infinitely many collisions with the focusing side in finite time \cite[Example, p.\ 298, Remark, p.\ 303]{Halpern}.

\cref{fig:threecircles2} shows three circles through $O$: two are osculating at $ O $ with the boundaries (with centers $C_+$ and $ C_-$), and  a ``middle'' circle  tangent to both boundaries at $ O $ and centered at  $C\dfn\frac12( C_++ C_-)$.  
\begin{figure}[htb]
	\center{  \includegraphics{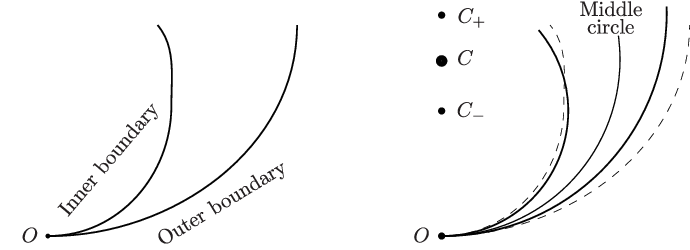}}
	\caption{A horn with centers of osculating circles and the ``middle'' circle.}  
	\label{fig:threecircles2}
\end{figure}
\begin{remark}\label{RemarkAngular}
The velocity of the particle in the direction of the tanget at $ O $ is not monontone: the collisions with opposite walls contribute opposite changes of this velocity.  Somewhat delicate estimates may seem to be needed to prove that the  repulsion from the cusp dominates.  This difficulty is circumvented by the key idea of our approach, which is to consider the {\it angular velocity} of the billiard particle with respect to $ C $,  the center of the middle circle. We will show that, in addition to the claim in the Theorem, this angular velocity   is a convex function of time, first decreasing and then increasing.
\end{remark}
The following lemma captures a key idea of the proof of \cref{thm:main}. 
\begin{lemma}\label{lem:key}
Let $ L $ be the angular momentum  of the particle {\it  relative to the center $C$ } of the middle circle. We think of the billiard particle as a point  of unit mass with unit speed and use the standard sign convention:    motion towards the cusp, i.e.,  clockwise, corresponds to $ L<0 $.

There exists a neighborhood of the cusp in which every collision with the boundary (whether focusing or dispersing) is repelling from the cusp in the sense of adding to $ L $ a counterclockwise component:
\begin{equation}
	 L_+> L_-, 
	\label{eq:L<L}
\end{equation}  
where $ L_- $ is the angular momentum before a collision and $ L_+ $ is the angular momentum after the same collision. Moreover, there exists a constant $ c>0 $ (independent of the trajectory) such that 
\begin{equation}
    L_+-L_- > cs \; \vec{v}_-\cdot \vec{n}, 
    \label{eq:deltaL}
\end{equation} 
where $s$ is the arc-length distance between the collision point and the cusp, $\vec v_-$ is the pre-collision velocity, $\vec n$ is the outward normal at the collision point, and ``$\cdot$'' denotes the dot product.
\end{lemma}

\vskip 0.1 in 
\begin{proof}[Proof of \cref{lem:key}]
We denote by $R_+, R_-$ the radii of the (osculating circles of the) sides and let $R\dfn\frac12(R_++R_-)$ (the radius of the middle circle) and $d\dfn\frac12(R_+-R_-)>0$.
\begin{figure}[htb]
	\center{  \includegraphics{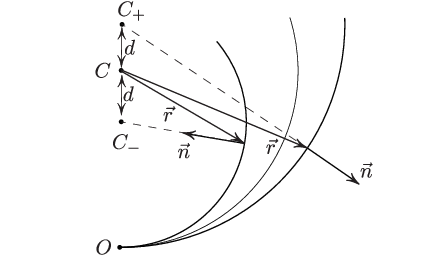}}
	\caption{Proof that the angular momentum relative to $ \overline O $ is repelled from the tip of the cusp.}  
	\label{fig:threecircles}
\end{figure}

With $C$ (the center of the middle circle) chosen as the origin, let $\vec r $  be the position vector of the particle at the moment of collision—see \cref{fig:threecircles}. The post- and pre-collision   velocity vectors $\vec v_\pm $ satisfy 
\begin{equation}
	     \vec v_+-\vec v_-=-2 (\vec v_-\cdot \vec n) \vec n, 
	\label{eq:collisionrule}
\end{equation}  
where $\vec n$ is the unit {\it  outward} normal to the  boundary. By definition,  $ L_\pm = \vec r \times \vec v_\pm\dfn\det[\vec r\ \vec v_\pm] $, the 2-dimensional cross product ($(x_1,y_1) \times (x_2,y_2) \dfn x_1y_2-x_2y_1$). Thus we obtain \eqref{eq:L<L}:
\begin{equation}
     L_+-L_- =
    \vec r \times (\vec v_+-\vec v_-) \buildrel{\mathclap{\eqref{eq:collisionrule}}}\over{=} -2 (\underbrace{\vec v_-\cdot \vec n}_{>0})\underbrace{\vec r \times \vec n}_{<0}  > 0.   
    \label{eq:deltaL1}
\end{equation} 
To complete the proof of \eqref{eq:deltaL}, it remains to show that 
\begin{equation}
     -2\vec r \times \vec n \geq cs 
    \label{eq:rtin}
\end{equation} 
for some $ c>0 $. 


We first prove this  for circles, concentrating on the focusing circle. Figure \ref{fig:horns3} yields
\begin{equation}
    -\vec r \times\vec n =| \vec r| \sin \alpha\ge R\sin \alpha,
   \label{eq:rtinn}
\end{equation}
so it remains to establish the lower bound $cs$ for $ \sin\alpha$.

To that end, we express
the length of the perpendicular from $C$ onto the long side of the triangle in \cref{fig:horns3}  in two ways: 
\begin{equation}
	 (R_+-d \cos \varphi ) \tan \alpha= d \sin \varphi .
	\label{eq:phi}
\end{equation}  
Using $\tan\alpha=\sin\alpha+O(\alpha^3)$,  $\sin\varphi=\varphi+O(\varphi^3)$, $\cos\varphi=1+O(\varphi^2)$, and $R_+-d=R$, \eqref{eq:phi} becomes 
$ R \sin  \alpha = d \varphi +O( \varphi ^2 ). $  
 But $\varphi  = s/R_++O(s^2)$ from \cref{fig:horns3}, so
 \begin{equation} 
   R\sin \alpha=dR_+^{-1}s+O(s^2).
    \label{eq:alpha}
\end{equation} 
Together with \eqref{eq:rtinn}, this proves \eqref{eq:rtin}, and hence (\ref{eq:deltaL}) by taking $c\dfn d/R_+>0$. 
The same argument applies to the dispersing boundary, and we omit it. 
 \begin{figure}[thb]
	\center{  \includegraphics[scale=0.8]{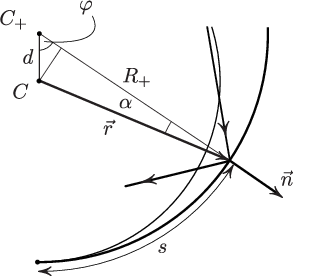}}
	\caption{The focusing boundary and the middle circle (radius $R= R_+-d$).}  
	\label{fig:horns3}
\end{figure}

This completes the proof of the Lemma for circles;   let us outline the proof for arbitrary curves. Let  $ \gamma (s) $ and $ \gamma_0(s) $ be parametric representations of the curve $ \gamma $ and its osculating circle $ \gamma_0 $, with arc-length counted from the tip of the cusp. 
We have 
\[
    \gamma (s) - \gamma_0(s) = O( s ^3 ), \ \ \gamma ^\prime  (s) - \gamma_0^\prime (s) = O( s ^2 ). 
\]  
These corrections are of higher order in $s$, which shows that (\ref{eq:deltaL}) still holds for osculating curves, for sufficiently small $ s $.  
 This completes the proof of Lemma \ref{lem:key}. 
\end{proof}
\begin{remark}\label{REMAngular2}
The main advantage of using angular momentum versus linear momentum is the following. As far as the linear momentum (in the direction of the tangent at the cusp) goes, there is competition: collisions with the dispersing boundary repel, while the collisions with the focusing boundary attract towards the cusp, and one is then forced into showing that the first effect dominates.  But considering angular momentum with respect to the center of an middle circle reveals a hidden cooperation: not just one but both boundaries contribute to repulsion from the cusp. 
\end{remark}
\begin{proof}[Proof of Theorem \ref{thm:main}]  Consider a trajectory starting in the neighborhood in \cref{lem:key}. Let $ L_{\pm}^n $ be the sequence of angular momenta at collisions while in this neighborhood. We have 
\[
 L_-^1<L_+^1=L_-^2<L_+^2\ldots. 
\]
We take this to be the maximal sequence and wish to show that it is finite. 
Assume the contrary: that it is infinite. A bounded infinite monotone sequence has a limit $L^\ast $, which in our case is in $ [-R, R] $. To arrive at a contradiction, consider two cases: 
\begin{enumerate}
\item $ L^\ast  \in [-R, 0] $, and 
\item $ L^\ast  \in (0,R] $. 
\end{enumerate} 

Note that in the first case, $ -R\leq L^\ast\leq 0 $, $  L^\ast =-R $ is impossible, since after the first collision the angular momentum strictly increases according to \eqref{eq:L<L}, and $-R $ is its least possible value. So $ -R< L_+^1<L_\pm^n\leq 0 $ for all $n>1$. Furthermore, the sequence $ \{\theta_n\} $ of polar angles (with respect to the middle circle) at the moments of collision is decreasing because $ L$ is negative, and thus so is the angular velocity (with respect to the center $ C $). Thus $ \theta_n \downarrow \theta^\ast\ge0 $. 

Consider first  $ \theta ^\ast > 0 $. There cannot be infinitely many successive collisions with the focusing side \cite{Halpern} (not to mention the dispersing side), so up to passing to a subsequence, we can take the $\theta_n$ to correspond to end-points of trajectory segments connecting the two boundaries. (In reality, the collisions become alternating after a possible succession of collisions with the focusing boundary.)
The segments approach the segment $ I $  of the radial line $ \theta= \theta ^\ast $ lying inside the cusp. Therefore each segment that is sufficiently close to $ I $     contributes a (uniformly) positive amount to the angular momentum, which implies that $ L_\pm^n \rightarrow \infty $, a contradiction. 

Next, we address the case of $ \theta ^\ast=0 $ (accumulation on the tip of the horn). Since $ 
L_\pm^n>L_+^1>-R $, the angles between   the trajectory and the lines $ \theta = {\rm const.} $ are bounded away from $\pm\pi/2$ by a positive constant (depending on $ L_+^1   = R \dot \theta (t_1+0)$, where $t_1$ denotes the time of the first collision). That is, trajectories are sufficiently transverse to the circles concentric with the middle circle. In particular, collisions alternate between the two boundaries while this transversality holds, i.e., for all sufficently large $ n $. And since the tangents of the two  boundaries and of the middle circle are arbitrarily close in sufficient proximity to the cusp, we conclude that $  | \vec v^n_\pm\cdot \vec n | \geq \varepsilon >0  $ for all sufficiently large $ n $. 
Using $ L_+^{n-1}=L_-^n $ in \eqref{eq:deltaL}, we   obtain
\begin{equation}\label{EQDeltaL}
L_+^n-L_+^{n-1}=L_+^n- L_-^n>\varepsilon cs_n.
\end{equation} 
It should be noted that $ s $ in equation \eqref{eq:deltaL}
is measured along a boundary, while here $s$ is measured along the middle circle; however, they are of the same order as long as $ | \vec v^n_\pm\cdot \vec n | \geq \varepsilon >0  $.

The sum of \eqref{EQDeltaL} over $n$ telescopes, and we obtain 
\[
L^*-L_+^1\ge  L^N_+-L^1_+ \ge \varepsilon c \sum_{n=2}^Ns_n,  
 \]  
so that $ s_n $ is summable; and since  $ s_n = R \theta_n $, we have  $ \sum_{n=1}^ \infty \theta_n < \infty $. 

On the other hand, the angles of intersection of our trajectory with the boundaries are all bounded away from $ 0 $. Together with the fact that the boundaries are tangent at the tip of the cusp this implies that there exists $ c_1>0 $ such that
\(
   0<\theta _k- \theta _{k+1}< c_1 \theta _k \theta_{k+1}, 
 \) 
hence 
\[
    \theta_{k+1}^{-1} - \theta_k^{-1} \leq c_1. 
\]  
Summation of this telescopes and gives
\(
     \theta_n ^{-1}  -  \theta_0 ^{-1}  \leq c_1n
\),  
so that \(\theta_n\ge\dfrac{1}{c_1n+\theta_0^{-1}}\), and thus
$
    \sum_{n=1}^ \infty  \theta_n = \infty,
$ 
a contradiction with summability of $\{\theta_n\}$.

The second case, $ L^\ast  \in (0,R] $, corresponds to escape from the horn: $ L_\pm^n>L_\pm^{n_0} >0$ for some $ n_0 $ and all $ n>n_0 $, so $ r(t) \dot \theta (t) >L_\pm^{n_0} >0 $ for all sufficiently large times $ t $. Since $ r(t) $ is bounded from above and there cannot be infinitely many collisions in finite time, we have $ \theta \rightarrow \infty $, a contradiction. 

This completes the proof of the theorem. 
\end{proof}
\section{Adiabatic invariant}\label{SAdiabatic}
It turns out that the product of particle's velocity normal to the cusp and the width of the cusp is an adiabatic invariant. This is similar the Fermi billiards, i.e. particles bouncing elastically between two slowly moving walls, the simplest imaginable model of ideal gas.
More precisely, we have the following. 
\begin{ltheorem}\label{THMAdiabatic}
$s ^2 \sqrt{1- \dot s ^2 }$ is  an adiabatic invariant for a billiard horn.
\end{ltheorem}
This is indeed proportional to the product of the width and the normal velocity, since  the width  $ \sim s ^2 $; this expression is  exactly  the same as in the concave cusp billiard \cite[p.\ 499]{BalintChernovDolgopyat}, \cite{Cusps}. (The latter article also notes implications for penetration depth.)

What we actually produce is the following continuum approximation for the motion into the cusp:
\begin{equation}
	 \ddot s = 2( 1 - \dot s ^2 ) /s. 
	\label{eq:seq1}
\end{equation}  
This differential equation preserves $s ^2 \sqrt{1- \dot s ^2 }$:
\[
\frac d{dt}s ^2 \sqrt{1- \dot s ^2 }=2s\dot s \sqrt{1- \dot s ^2 }-s^2\frac{\dot s\ddot s}{\sqrt{1- \dot s ^2 }}\overset{\eqref{eq:seq1}}{=}0.
\]
For the actual  billiard, $s ^2 \sqrt{1- \dot s ^2 }$ is only approximately  conserved (in a sense we do not examine here); that is, it is  an adiabatic invariant for our billiard map.  

We now produce the continuum approximation \eqref{eq:seq1} for the motion into the cusp as follows. We first compute the change of angular momentum $ \Delta L $ in terms of $ L $ and $ s $ after consecutive collisions with opposite boundaries. The resulting difference equation \eqref{eq:pre} for $ L $ has as its continuum limit  the differential equation \eqref{eq:seq1} for $ s $, which  preserves $s ^2 \sqrt{1- \dot s ^2 }$, the adiabatic invariant we seek. 

Now, $ \Delta L $ is given by   \eqref{eq:deltaL1}, where the first factor is $ \vec v_-\cdot\vec n= \cos \psi_\partial $, with $ \psi _\partial $  the angle between the trajectory and the normal to the boundary.  But $\psi_\partial= \psi + O(s ) $, where $ \psi $ is the angle between the trajectory segment and the middle circle, and $ s $ is the arc-length along the middle circle. Using  $  L=R\sin  \psi $ in the last step,  we have 
\begin{equation}
	   \vec v_-\cdot\vec n= \cos \psi_\partial  = \sqrt{  1- \sin  ^2 \psi_\partial } = \sqrt{  1- \frac{L ^2 }{R ^2 } } +O(s).
	\label{eq:vdotn}
\end{equation}

To decipher the second factor $ -{\vec r} \times  {\vec n} = |\vec r| \sin  \alpha  $ in (\ref{eq:deltaL1}) (see \eqref{eq:rtinn}),  we note that $  | \vec r  | = R+ O(s^2) $, while $ \sin  \alpha =ds/R_+R+O( s ^2 )$   by \eqref{eq:alpha}; this results in
\begin{equation}\label{EQRcrosN}
    -{\vec r} \times  {\vec n} =-Rds/(RR_+)+ O(s ^2 )= {d}R_+^{\;-1}s + O(s ^2 ). 
\end{equation}  
Substituting this and \eqref{eq:vdotn}  into \eqref{eq:deltaL1} gives
\begin{equation}
	 \Delta L_{\rm outer} = -2 (\vec v_-\cdot \vec n)\;\vec r \times \vec n\Buildrel{\Tiny\eqref{eq:vdotn}\eqref{EQRcrosN}}{=}2\sqrt { 1- \frac{L ^2 }{R ^2 } }dR_+^{\;-1}s + O(s ^2 ). 
	\label{eq:dl1}
\end{equation}  
The time of the round trip from the outer wall and back   is 
\begin{equation}
	 \Delta t = 2 \frac{\rm width}{\cos \psi }   = 
     2 \frac{\frac{1}{2} (R_-^{-1}- R_+^{\;-1})s ^2  +O(s^3)}{\cos \psi }  \Buildrel{\tiny\eqref{eq:vdotn}}= 
     \frac{ (R_- ^{-1}- R_+ ^{\;-1})s ^2 +O(s^3) }{\sqrt{  1- \frac{L ^2 }{R ^2 } } }  . 
	\label{eq:Deltat}
\end{equation}  
For the  collision with the \emph{outer wall}, dividing \eqref{eq:dl1} by \eqref{eq:Deltat} gives  
\begin{equation}
	   \frac{\Delta L_{\rm outer}}{\Delta t}  ={2\frac{d}{{R_+}(R_- ^{-1}- R_+ ^{\;-1})}}\biggl( 1- \frac{L ^2 }{R ^2 }  \biggl)  
     \frac{1}{s}+ O(s^0)=R_-\biggl( 1- \frac{L ^2 }{R ^2 }  \biggl)  
     \frac{1}{s}+ O(s^0); 
	\label{eq:DLt2}
\end{equation} 
here we added a subscript as reminder that the collision is with the outer boundary. The order of the error is explained by the calculation $\frac{s(a+O(s))}{s^2(b+O(s))}=\frac{a}{bs}+O(s^0)$. 

For the collision with the inner boundary, likewise 
\begin{equation}
	  \frac{\Delta L_{\rm inner}}{\Delta t}  =R_+\biggl( 1- \frac{L ^2 }{R ^2 }  \biggl)  
     \frac{1}{s}+ O(s^0).
	\label{eq:dlo2}
\end{equation} 
Adding \eqref{eq:DLt2} and  \eqref{eq:dlo2}, we obtain the estimate of the change of angular momentum per unit time: 
\begin{equation}
	  \frac{\Delta L}{\Delta t} = 2R  \biggl( 1- \frac{L ^2 }{R ^2 }  \biggl)\frac{1}{s}+   O(s^0).
	\label{eq:pre}
\end{equation}    
Consider now our billiard  particle as being shadowed by a smoothly moving particle sliding along the middle circle and undergoing the same rate of change of angular momentum. 
That is, we approximate (\ref{eq:pre}) by the differential equation 
\[
\frac{d L}{d t} = 2R  \biggl( 1- \frac{L ^2 }{R ^2 }  \biggr)\frac{1}{s}. 
\]  
Since 
$ L  = R \dot s$, where $ s $ measures the arc-length along the middle circle (or along either of the two boundaries, since this is only an approximation valid for small $ s $), we obtain 
\[
    R \ddot s = \frac{dL}{dt} = 2R  \biggl( 1- \frac{L ^2 }{R ^2 }\biggr)\frac{1}{s}= 2R( 1- \dot s^2)/s,
\] 
hence \eqref{eq:seq1}.


\bibliography{cusp-bib}

\end{document}